\documentclass[12pt,oneside,reqno,bibliography=totoc]{scrartcl}

\usepackage{cihan}

\author{Cihan Bahran}
\affil{Department of Mathematics, Bo\u{g}azi\c{c}i University\\
Bebek, 34342 Istanbul, Turkey}
\date{}

\AtBeginDocument{%
  \addtolength\abovedisplayskip{-0.5\baselineskip}%
  \addtolength\belowdisplayskip{-0.5\baselineskip}%
}

\title{Onset of regularity for $\FI$-modules}

\newcommand{\FI}{\mathbf{FI}}

\DeclareMathOperator{\crit}{crit}

\newcounter{dummy}
\makeatletter
\newcommand\sitem[1][]{\item[(#1)]\refstepcounter{dummy}\def\@currentlabel{#1}}
\makeatother

\newcommand{\cofi}[1]{\co_{#1}^{\FI}}

\newcommand{\cofibol}[1]{\mathbf{H}_{#1}^{\FI}}

\DeclareMathOperator{\reg}{reg}

\newcommand{\tgen}{t_{0}}
\newcommand{\trel}{t_{1}}
\newcommand{\shift}[2]{\mathbf{\Sigma}^{#2}   #1 }

\newcommand{\hell}{h}
\newcommand{\local}{\hell^{\text{max}}}

\DeclareMathOperator{\cone}{cone}

\newcommand{\locoh}[1]{\co_{\mathfrak{m}}^{#1}}

\newcommand{\comm}{\mathbf{k}}

\AtBeginDocument{%
   \def\MR#1{}
}

\newcommand{\specht}[2]{\operatorname{S}_{#1}(#2)}

\newcommand{\loct}[1]{#1\!\left[\frac{1}{2}\right]}
\newcommand{\locr}[1]{#1\!\left[\frac{1}{3}\right]}
\newcommand{\locl}[1]{#1\!\left[\frac{1}{\ell}\right]}

\makeatletter
\def\blfootnote{\gdef\@thefnmark{}\@footnotetext}
\makeatother

\begin{document}
\maketitle

\blfootnote{\textup{2010} \textit{Mathematics Subject Classification}.
18A25.} 
\blfootnote{The author was supported in part by T\"{U}B\.{I}TAK 119F422.}
\vspace{-0.8 in}
\begin{onecolabstract} 
In terms of local cohomology, we give an explicit range as to when the $\FI$-homology of an $\FI$-module attains the degree predicted by its regularity.
\end{onecolabstract}

{\tableofcontents}

\section{Introduction}
Let $\comm$ be a commutative ring. An $\FI$-module over $\comm$ is a functor $\FI \rarr \lMod{\comm}$ where $\FI$ is the category of finite sets and injections. The category of $\FI$-modules over $\comm$ is the functor category $[\FI,\lMod{\comm}]$. Given $W \colon \FI \rarr \lMod{\comm}$, we write 
\begin{align*}
 \deg(W) &:= \min\{d \geq -1 : W_{S} = 0 \text{ for } |S| > d\} 
 \\
 &\in \{-1,0,1,2,3,\dots\} \cup \{\infty\} \, .
\end{align*}

\paragraph{$\FI$-homology and regularity.} A systematic way of investigating the ``unstable'' part of $V \colon \FI \rarr \lMod{\comm}$ is to study its quotient $\cofi{0}(V)$ defined by
\begin{align*}
 \cofi{0}(V)_{S} := \coker\! \left(
 \bigoplus_{A \subsetneq S} V_{A} \rarr V_{S}
 \right) \, .
\end{align*}
The functor $\cofi{0} \colon [\FI,\lMod{\comm}] \rarr [\FI,\lMod{\comm}]$ is right exact and has left derived functors $\cofi{i} := \operatorname{L}_{i}\!\cofi{0}$. Here one can think of 
\begin{itemize}
 \item $\cofi{0}(V)$ as the ``generators'' of $V$,
 \vspace{0.1cm}
 \item $\cofi{1}(V)$ as the ``relations'' of $V$,
 \vspace{0.1cm}
 \item $\cofi{i}(V)$ for $i \geq 2$ as the corresponding higher syzgies.
\end{itemize}
Setting $t_{i}(V) : = \deg\!\left(\cofi{i}(V)\right)$, we say $V$ is \textbf{presented in finite degrees} if $\tgen(V)$ and $\trel(V)$ are finite. We also set the \textbf{regularity} of $V$ as
\begin{align*}
 \reg(V) &:= \max\!\left\{t_{i}(V) - i : i \geq 1\right\} 
 \\
 &\in \{-2,-1,0,1,\dots\} \cup\{\infty\} \, ,
\end{align*}
which is finite whenever $V$ is presented in finite degrees \cite[Theorem A]{ce-homology}.

\paragraph{Torsion and local cohomology.}
We say $V \colon \FI \rarr \lMod{\comm}$ is \textbf{torsion} if for every finite set $S$ and $x \in V_{S}$, there is an injection $\alpha \colon S \emb T$ such that $V_{\alpha}(x) = 0 \in V_{T}$. We write
\begin{align*}
 \locoh{0} \colon [\FI,\lMod{\comm}] \rarr [\FI,\lMod{\comm}]
\end{align*}
for the functor which assigns an $\FI$-module over $\comm$ its largest torsion $\FI$-submodule; it is left exact. For each $j \geq 0$, we write $\locoh{j} := \operatorname{R}^{j}\!\locoh{0}$ for the $j$-th right derived functor of $\locoh{0}$, and write 
\begin{align*}
 h^{j}(V) &:= \deg(\locoh{j}(V)) 
 \\
 &\in \{-1,0,1,\dots\} \cup \{\infty\} \, ,
\end{align*}
for every $\FI$-module $V$.

The equality statement in the following result (proved in \cite{nss-regularity}) gives a precise relationship between the $t_{i}(V)$'s and $h^{j}(V)$'s, reminiscent of (and motivated by) the relationship between the syzgyies and the local cohomology of a graded module over a polynomial ring in commutative algebra \cite{eisenbud-goto}.

\begin{thm}[{\cite[Theorem 21]{gan-shift-seq}, \cite[Theorem 1]{nss-regularity}}] \label{nss-main}
Let $\comm$ be a commutative ring, and  $V \colon \FI \rarr \lMod{\comm}$ be presented in finite degrees which is not $\cofi{0}$-acyclic. Then 
we have a weakly increasing sequence
\begin{align*}
  0 \leq t_{1}(V) - 1 \leq t_{2}(V) - 2 \leq t_{3}(V) - 3 \leq \cdots
\end{align*}
that stabilizes at 
\begin{align*}
 \reg(V) &= \max\{h^{j}(V) + j : \locoh{j}(V) \neq 0\} \\
 &= \max\{h^{j}(V) + j : h^{j}(V) \geq 0\} < \infty \, .
\end{align*}
\end{thm}

The main result of this paper is an explicit range for the stabilization of the weakly increasing sequence in Theorem \ref{nss-main} in terms of local cohomology. To that end, we need to give one more definition.

\paragraph{Critical index.} Given $V \colon \FI \rarr \lMod{\comm}$ presented in finite degrees which is not $\cofi{0}$-acyclic, we define
\begin{align*}
 \crit(V) := 
 \min\!\left\{
 j : h^{j}(V) \geq 0 \text{ and }
 h^{j}(V) + j = \reg(V)
 \right\}
\end{align*}
which is well-defined by Theorem \ref{nss-main}, and call it the \textbf{critical index} of $V$.

\begin{thmx} \label{main}
 Let $\comm$ be a commutative ring, and $V \colon \FI \rarr \lMod{\comm}$ be presented in finite degrees which is not $\cofi{0}$-acyclic. Then writing $\gamma := \crit(V)$, we have $t_{a}(V) - a = \reg(V)$ whenever 
\begin{align*}
 a \geq 
\begin{cases}
 \max\{1,\,h^{\gamma}(V) - \gamma\} & \text{if $2 \in \comm^{\times}$ is invertible,}
 \\
 \max\{1,\,2h^{\gamma}(V) - \gamma\} & \text{otherwise.}
\end{cases}
\end{align*}
\end{thmx}


We note two consequences of Theorem \ref{main} which have bounds in terms of more commonly used invariants of $\FI$-modules.

\begin{corx} \label{local-terms}
 Under the hypotheses of Theorem \ref{main}, writing 
\begin{align*}
 \local(V) := \max\{h^{j}(V) : j \geq 0\}
\end{align*}
which is the local degree in \emph{\cite[Definition 2.12]{cmnr-range}}, we have $t_{a}(V)-a = \reg(V)$ whenever
\begin{align*}
 a \geq 
\begin{cases}
 \max\{1,\,\local(V)\} & \text{if $2 \in \comm^{\times}$ is invertible,}
 \\
 \max\{1,\,2\local(V)\} & \text{otherwise.}
\end{cases}
\end{align*}
\end{corx}
\begin{proof}
 By the definition of critical degree, we have $\gamma := \crit(V) \geq 0$ and hence 
\begin{align*}
 h^{\gamma}(V) - \gamma &\leq h^{\gamma}(V) \leq \local(V) \, ,
 \\
 2h^{\gamma}(V) - \gamma &\leq 2h^{\gamma}(V) \leq 2\local(V) \, .
\end{align*}
Now apply Theorem \ref{main}.
\end{proof}

\begin{corx} \label{pres-terms}
 Under the hypotheses of Theorem \ref{main}, assume $\tgen(V) \leq g$ and $\trel(V) \leq r$. Then we have $t_{a}(V)-a = \reg(V)$ whenever
\begin{align*}
 a \geq 
\begin{cases}
 \max\{1,\,g + r  - 1\} & \text{if $2 \in \comm^{\times}$ and $g < r$,}
 \\
 \max\{1,\,2r - 2\} & \text{if $2 \in \comm^{\times}$ and $g \geq r$,}
 \\   
 \max\{1,\,2g + 2r-2\} & \text{if $2 \notin \comm^{\times}$ and $g < r$,}
 \\ 
 \max\{1,\,4r - 4\} & \text{if $2 \notin \comm^{\times}$ and $g \geq r$.} 
\end{cases}
\end{align*}
\end{corx}
\begin{proof}
 First, as we are assuming $V$ is not $\cofi{0}$-acyclic, we have $\tgen(V) \geq 0$ (as $V \neq 0$) and $\trel(V) \geq 0$. Next, by the definition of critical degree, we have $\gamma := \crit(V) \geq 0$ and $h^{\gamma}(V) \geq 0$, and hence by Theorem \ref{nss-main},
\begin{align*}
 h^{\gamma}(V) - \gamma &\leq h^{\gamma}(V) + \gamma \leq \reg(V) \, ,
 \\
 2h^{\gamma}(V) - \gamma &\leq 2h^{\gamma}(V) + 2\gamma \leq 2\reg(V) \, .
\end{align*}
Now we may apply \cite[Theorem A]{bahran-reg} to bound $\reg(V)$ and invoke Theorem \ref{main}.
\end{proof}

\begin{conj}
 The bound conditional on the assumption $2 \in \comm^{\times}$ in Theorem \ref{main} (and hence those in Corollary \ref{local-terms} and Corollary \ref{pres-terms}) holds unconditionally.
\end{conj}

\section{Good ideals and the invariant $\nu_{I}$}
In this section we review the notion of a good ideal in the group algebra of a symmetric group and the invariant $\nu$ attached to such an ideal, defined in \cite[Section 3]{nss-regularity}. For clarity we will use the notation $\nu_{I}$ to emphasize the dependence on the choice of ideal $I$.

\begin{defn} \label{gi-defn}
 Let $\comm$ be a commutative ring, $\eps_{a}$ be the sign $\comm\sym{a}$-module, and $\ell \geq 2$. A two-sided ideal $I \subseteq \comm\sym{\ell}$ is called \textbf{good} if the following hold: 
\vspace{-0.2cm}
\begin{birki}
 \item $I$ is idempotent.
 \item $I$ annihilates $\eps_{\ell}$.
 \item $I$ does not annihilate $\Ind_{\sym{\ell-1}}^{\sym{\ell}}(\eps_{\ell-1}) \otimes_{\comm} E$ for every nonzero $\comm$-module $E$.
 \item $I$ is $\comm$-flat. 
\end{birki}
\end{defn}

\begin{defn} \label{nu-defn}
 Let $\comm$ be a commutative ring, $\ell \geq 2$, and $I \subseteq \comm\sym{\ell}$ be a good ideal. Then for every $d \geq 0$  and a $\comm\sym{d}$-module $M$, we set
\begin{align*}
 \nu_{I}(M) := 
\begin{cases}
 d+1 & \text{if $M = 0$,}
 \\
 d - \max\!\left\{
 0 \leq r \leq \floor*{\frac{d}{\ell}} :
 I^{\boxtimes r} \nsubseteq \ann_{\comm\sym{\ell}^{\times r}}(M)
 \right\}
 & \text{if $M \neq 0$.}
\end{cases}
\end{align*}
\end{defn}
\begin{rem} \label{nss-nu-compare}
 In the notation of \cite[Section 3]{nss-regularity}, the two-sided ideal $I_{d}(r)$ of $\comm\sym{d}$ is defined as 
\begin{itemize}
\item $0$ if $\ell r > d$, and
\item the two-sided ideal of $\comm\sym{d}$ generated by $\alpha_{r}(I^{\boxtimes r})$, where $\alpha_{r} \colon \comm\sym{\ell}^{\times r} \emb \comm\sym{d}$ is the usual inclusion, if $\ell r \leq d$.
\end{itemize}
Note that if $r \geq 1$ and $\ell r \leq d$, writing $\iota_{r-1} \colon \comm\sym{\ell}^{\times (r-1)} \emb \comm\sym{\ell}^{\times r}$ for the inclusion, we have $\iota_{r-1}(I^{\boxtimes (r-1)}) \supseteq I^{\boxtimes r}$ and hence $\alpha_{r-1}(I^{\boxtimes (r-1)}) = \alpha_{r}\iota_{r-1}(I^{\boxtimes (r-1)}) \supseteq \alpha_{r}(I^{\boxtimes r})$. Therefore $\comm\sym{d}$ has a decreasing chain of two-sided ideals 
\begin{align*}
 \comm\sym{d} = I_{d}(0) \supseteq \cdots \supseteq I_{d}(\floor{d/\ell}) \supseteq 0 \supseteq \cdots
\end{align*}
so that 
\begin{align*}
 &\,\min\!\left\{r : I_{d}(r) \subseteq \ann_{\comm \sym{d}}(M) \text{ for } s>r \right\}
 \\
 &= 
\begin{cases}
 -1 & \text{if $M = 0$,}
 \\
 \max\!\left\{0 \leq r \leq \floor*{\frac{d}{\ell}} : I_{d}(r) \nsubseteq \ann_{\comm \sym{d}}(M) \right\}
 & \text{if $M \neq 0$},
\end{cases}
 \\
 &=
\begin{cases}
 -1 & \text{if $M = 0$,}
 \\
 \max\!\left\{0 \leq r \leq \floor*{\frac{d}{\ell}} : I^{\boxtimes r} \nsubseteq \ann_{\comm\sym{\ell}^{\times r}}(M) \right\}
 & \text{if $M \neq 0$},
\end{cases}
\end{align*}
showing that Definition \ref{nu-defn} matches \cite[Definition 3.3]{nss-regularity} for $M \neq 0$.\footnote{Strictly speaking, \cite[Definition 3.3]{nss-regularity} leaves $\nu(0)$ undefined. We will point out in the necessary places that our convention $\nu_{I}(0)= d+1$ for the zero $\comm\sym{d}$-module does not cause any trouble.}
\end{rem}

\begin{prop} \label{nu-props}
 Let $\comm$ be a commutative ring, $\ell \geq 2$, and $I \subseteq \comm\sym{\ell}$ be a good ideal. Then for every $d \in \nn$ and $\comm\sym{d}$-modules $L$, $M$, $N$, the following hold:
\begin{birki}
 \item If $0 \rarr L \rarr M \rarr N \rarr 0$ is a short exact sequence of $\comm\sym{d}$-modules, then 
\begin{align*}
  \nu_{I}(M) = \min\{\nu_{I}(L),\nu_{I}(N)\} \, .
\end{align*}
 \item If $M \neq 0$, then writing $\eps_{a}$ for the sign $\comm \sym{a}$-module, we have
\begin{align*}
 \nu_{I}\!\left(
 \Ind_{\sym{d} \times \sym{a}}^{\sym{d+a}}\!\left(
 M \boxtimes_{\comm} \eps_{a}
 \right)
 \right) = a
\end{align*}
 whenever $a \geq (\ell-1)d$.
\item For $d \geq 1$, we have $\nu_{I}(M) - 1 \leq \nu_{I}\!\left(
 \Res_{\sym{d-1}}^{\sym{d}}(M)
 \right) \leq \nu_{I}(M)$, and more specifically
\begin{align*}
 \nu_{I}\!\left(
 \Res_{\sym{d-1}}^{\sym{d}}(M)
 \right) =
\begin{cases}
\nu_{I}(M) & \text{if $\nu_{I}(M) = d - \frac{d}{\ell}$,}
\\
\nu_{I}(M) - 1 &\text{otherwise.}
\end{cases}
\end{align*}
\end{birki}
\end{prop}
\begin{proof}
 (1): If $L,M,N$ are all nonzero, this is \cite[Proposition 3.4]{nss-regularity}. If $L = 0$, then $M \cong N$ so that $\nu_{I}(M) = \min\{\nu_{I}(N), d+1\}$ because the $\nu_{I}$ of a $\comm\sym{d}$-module is always $\leq d+1$. If $N=0$, a similar argument works. If $M=0$, then $L=N=0$ and both sides of the equation are equal to $d+1$.
 
(2): This is \cite[Proposition 3.6]{nss-regularity} (the induction functor $\Ind_{\sym{d} \times \sym{a}}^{\sym{d+a}}$ in front of the external tensor product is missing in the statement of \cite[Proposition 3.6]{nss-regularity}, as can be deduced from the right hand side of the Mackey decomposition in its proof).
 
(3): By definition, if $M \neq 0$ and $\nu_{I}(M) \neq d- \frac{d}{\ell}$, then we have $\nu_{I}(M) \nleq d - \frac{d}{\ell}$ and hence $\nu_{I}(M) \nleq \floor*{d - \frac{d}{\ell}} = d - \ceil*{\frac{d}{\ell}}$ and
\begin{align*}
 d - \nu_{I}(M) &= \max\!\left\{
 0 \leq r \leq \floor*{\frac{d}{\ell}} :
 I^{\boxtimes r} \nsubseteq \ann_{\comm\sym{\ell}^{\times r}}(M)
 \right\} \leq \ceil*{\frac{d}{\ell}}-1
 < \frac{d}{\ell} \, ,
 \\
 d- \nu_{I}(M) &\leq \floor*{\frac{d-1}{\ell}} \, .
\end{align*} 
Thus
\begin{align*}
 d - \nu_{I}(M) &= \max\!\left\{
 0 \leq r \leq \floor*{\frac{d-1}{\ell}} :
 I^{\boxtimes r} \nsubseteq \ann_{\comm\sym{\ell}^{\times r}}(M) 
 \right\}
 \\
 &= d-1-\nu_{I}\!\left(
 \Res_{\sym{d-1}}^{\sym{d}}(M)
 \right) \, ,
\end{align*}
as desired. If $\nu_{I}(M) = d- \frac{d}{\ell}$, then writing $r := d/\ell$, we have 
\begin{align*}
 I^{\boxtimes r} \nsubseteq \ann_{\comm\sym{\ell}^{\times r}}(M)
\end{align*}
and hence
\begin{align*}
 I^{\boxtimes (r-1)} \nsubseteq \ann_{\comm\sym{\ell}^{\times (r-1)}}(M)
\end{align*}
as well, where $\floor*{\frac{d-1}{\ell}} = r-1$. Consequently 
\begin{align*}
 d-1-\nu_{I}\!\left(
 \Res_{\sym{d-1}}^{\sym{d}}(M)
 \right) = r-1 \, ,
\end{align*}
as desired. If $M=0$, then $ \nu_{I}\!\left(
 \Res_{\sym{d-1}}^{\sym{d}}(M)
 \right) = d = \nu_{I}(M) - 1$ again.
\end{proof}

\begin{rem}[Minimum number of rows]
As alluded to in \cite[beginning of Section 3, Remark 3.7]{nss-regularity}, the invariant $\nu_{I}$ is an attempt to define ``minimum number of rows in a simple object'' for a general ring $\comm$. Indeed if $\comm$ is a field of characteristic zero and $M$ a $\comm\sym{d}$-module, writing $\specht{\comm}{\lambda}$ for the Specht $\comm\sym{d}$-module associated to a partition $\lambda \vdash d$ and
\begin{align*}
 \Lambda(M) := \left\{
 \lambda \vdash d :
 \specht{\comm}{\lambda} \text{ is a summand of } M \right\} \, ,
\end{align*}
the invariant  
\begin{align*}
 \nu(M) := 
\begin{cases}
 d+1 & \text{if $M=0$,}
 \\
 \min\!\left\{
 \text{number of rows in $\lambda$} : \lambda \in \Lambda(M) \right\} 
 &\text{otherwise,}
\end{cases}
\end{align*}
satisfies part (1) of Proposition \ref{nu-props} as written, part (2) of Proposition \ref{nu-props} with $\ell = 2$ by Pieri's rule, and 
\begin{align*}
 &\,\,\nu\!\left(
 \Res_{\sym{d-1}}^{\sym{d}}(M)
 \right) 
 \\
 &=
\begin{cases}
\nu(M) & \text{if $M \neq 0$ and every row of every $\lambda \in \Lambda(M)$ has size $\geq 2$,}
\\
\nu(M) - 1 &\text{otherwise.}
\end{cases}
\end{align*}
Note that the first case above cannot occur if $\nu(M) > d/2$, giving a similar statement to part (3) of Proposition \ref{nu-props} with $\ell = 2$.
\end{rem}

\section{$\FI$-modules}

\paragraph{The shift functor.} Given $V \colon \FI \rarr \lMod{\comm}$,  we write $\shift{V}{}$ for the composition 
\begin{align*}
 \FI \xrightarrow{- \sqcup \{*\}} \FI \xrightarrow{V} \lMod{\comm} \, .
\end{align*}
The main property of the shift is that as an $\sym{n}$-representation, we have
\begin{align*}
 (\shift{V}{})_{n} =\Res_{\sym{n}}^{\sym{n+1}} V_{n+1} \, .
\end{align*}
If $\deg(V) < \infty$ and $V \neq 0$, we also have $ \deg(\shift{V}{}) = 
\deg(V) - 1$.

\begin{prop} \label{shift-stuff}
 Let $\comm$ be a commutative ring, and $V \colon \FI \rarr \lMod{\comm}$ be presented in finite degrees such that neither $V$ nor $\shift{V}{}$ is $\cofi{0}$-acyclic. Writing $\gamma := \crit(V)$, $\rho := \reg(V)$, we have 
\begin{align*}
 \reg(\shift{V}{}) \leq \rho-1 \quad \text{and} \quad
 \crit(\shift{V}{}) \leq \gamma \, .
\end{align*}
Moreover, the following are equivalent:
\begin{birki}
\item $\reg(\shift{V}{}) < \rho-1$.
\vspace{0.1cm}
\item $h^{\gamma}(V) = 0$.
\vspace{0.1cm}
\item $\{j : h^{j}(V) \geq 0\} \subseteq [0,\,\gamma]$ and $\{j : h^{j}(\shift{V}{}) \geq 0\} \subseteq [0,\,\gamma-1]$.
\vspace{0.1cm}
\item $\crit(\shift{V}{}) < \gamma$.
\end{birki}
\end{prop}
\begin{proof}
By Theorem \ref{nss-main} and the fact that the shift functor $\shift{}{}$ commutes with local cohomology $\locoh{\star}$ (for instance by \cite[proof of Proposition 2.6]{bahran-polynomial} or in this context by \cite[Theorem 2.10]{cmnr-range}), we have 
\begin{align*}
 \reg(\shift{V}{}) &= \max\{h^{j}(\shift{V}{}) + j : \shift{\!\locoh{j}(V)}{} \neq 0\} \\
 &= \max\{h^{j}(V) + j - 1 : h^{j}(V) \geq 1\} 
 \\
 &\leq \max\{h^{j}(V) + j - 1 : h^{j}(V) \geq 0\} = \rho-1 \, .
\end{align*}
We also see from here that because $h^{\gamma}(V) + \gamma - 1 = \rho-1$, (1) implies that $h^{\gamma}(V) \ngeq 1$ and hence (2). Next, assume (2). Then $\locoh{j}(V) \neq 0$ implies by Theorem \ref{nss-main} that 
\begin{align*}
 0 \leq j \leq h^{j}(V) + j \leq \rho = h^{\gamma}(V) + \gamma = \gamma \, ,
\end{align*}
and because $h^{j}(\shift{V}{}) \geq 0$ if and only if $h^{j}(V) \geq 1$, we get (3). It is evident that (3) implies (4). Finally, (4) implies that if we write $\beta := \crit(\shift{V}{})$, then by the equation we established for $\reg(\shift{V}{})$ in the beginning of the proof and the definition of $\gamma = \crit(V)$, we have
\begin{align*}
 \reg(\shift{V}{}) &= \max\{h^{j}(V) + j - 1 : h^{j}(V) \geq 1\}
 \\
 &= h^{\beta}(V) + \beta - 1 < h^{\gamma}(V) + \gamma - 1 = \rho - 1 \, .
\end{align*}
Having established the equivalence of (1)-(4), to show $\crit(\shift{V}{}) \leq \gamma$ in general, we may assume (2) does not hold, that is, $h^{\gamma}(V) \geq 1$. In this case $\reg(\shift{V}{}) = \rho-1$ and
\begin{align*}
 \crit(\shift{V}{}) = 
  \min\!\left\{
 j : h^{j}(V) \geq 1 \text{ and }
 h^{j}(V) + j - 1 = \rho-1
 \right\} = \gamma \, .
\end{align*}
\end{proof}

\begin{thm}[\cite{nss-regularity}] \label{nss-nu}
Let $\comm$ be a commutative ring, $\ell \geq 2$, and $I \subseteq \comm\sym{\ell}$ be a good ideal. Then for $V \colon \FI \rarr \lMod{\comm}$ presented in finite degrees which is not $\cofi{0}$-acyclic, writing $\gamma := \crit(V)$, $\rho := \reg(V)$, and
\begin{align*}
 \gamma_{a}(V) :=  \nu_{I}\!\left(
 \cofi{a}(V)_{a + \rho}
 \right) - a
 \, ,
\end{align*}
we have $\gamma_{a}(V) = \gamma$ for $a \gg 0$.
\end{thm}
\begin{proof}
 By \cite[Theorem 2.10]{cmnr-range} there is a bounded cochain complex $F$ of $\cofi{0}$-acyclic $\FI$-modules generated in finite degrees, and a chain map $\alpha\colon V \rarr F$ whose mapping cone $\cone(\alpha)$ is exact in high enough degrees (consider $V$ as a cochain complex concentrated in degree 0), that is, there exists $N \in \nn$ such that
\begin{align*}
 \deg\!\left(\co^{j}(\cone(\alpha))\right) \leq N \,\, \text{for every}
 \,\, j \, .
\end{align*}
Thus, with the notation $(-)_{\leq N}$ described in \cite[Lemma 2.3]{nss-regularity}, the natural map 
\begin{align*}
 \cone(\alpha) \rarr \cone(\alpha)_{\leq N}
\end{align*}
is a quasi-isomorphism. Consequently, the bounded cochain complex
\begin{align*}
 T := \cone(\alpha)_{\leq N}[1]
\end{align*}
 (which is suppported in cohomological degrees $\geq 0$ and $\FI$-degrees $\leq N$) sits in an exact triangle $T \rarr V \rarr F \rarr$ in the (bounded) derived category of $\FI$-modules (see \cite[Theorem 2.4]{nss-regularity}). We can now argue 
\begin{itemize}
 \item as in \cite[Proposition 2.6]{nss-regularity} to conclude $\locoh{j}(V) \cong \co^{j}(T)$ for every $j$, 
,
 \item as in \cite[proof of Theorem 1.1]{nss-regularity} to conclude $\cofi{a}(V) \cong \cofibol{a}(T)$ for every $a \geq 1$ (here $\cofibol{*}$ is the left hyper-derived functor of $\cofi{0}$ \cite[Definition 5.7.4]{weibel-hom-alg})

\end{itemize}
Now because \cite[Proposition 4.3]{nss-regularity} applies to every bounded $\FI$-complex of finite degree, the number 
\begin{align*}
 \rho &=
 \max\{h^{j}(V) + j : \locoh{j}(V) \neq 0\}
 \\
 &=
 \max\{\deg(\co^{j}(T)) + j : \co^{j}(T) \neq 0\}  \, ,
\end{align*} 
where the first equality holds by Theorem \ref{nss-main}, satisfies
\begin{align*}
 \gamma_{a}(V) = \nu_{I}\!\left(\cofibol{a}(T)_{a+\rho}\right) - a = \gamma
\end{align*}
for $a \gg 0$.
\end{proof}
\subsection{Explicit bounds}
We obtain essentially all our explicit bounds in this section.
\begin{thm} \label{sharp-nu}
Let $\comm$ be a commutative ring, $\ell \geq 2$, and $I \subseteq \comm\sym{\ell}$ be a good ideal. Then for $V \colon \FI \rarr \lMod{\comm}$ presented in finite degrees which is not $\cofi{0}$-acyclic, writing $\gamma := \crit(V)$, $\rho := \reg(V)$, and
\begin{align*}
 \gamma_{a}(V) :=  \nu_{I}\!\left(
 \cofi{a}(V)_{a + \rho}
 \right) - a
 \, ,
\end{align*}
we have a weakly decreasing sequence
\begin{align*}
  \rho+1 \geq \gamma_{1}(V) \geq \gamma_{2}(V) \geq \gamma_{3}(V) \geq \cdots
\end{align*}
that stabilizes at $\gamma$ such that either $\gamma_{1}(V) = \gamma$ immediately, or otherwise 
\begin{align*}
 \max\!\left\{a \geq 1:
 \gamma_{a}(V) > \gamma
 \right\} = (\ell-1)h^{\gamma}(V)-\gamma-1 \, .
\end{align*}
\end{thm}
\begin{proof}
 By the definition of $\nu_{I}$ applied to a $\comm\sym{1+\rho}$-module, we have 
$ \nu_{I}\!\left(
 \cofi{1}(V)_{1 + \rho}
 \right) \leq \rho+2$, 
 so $\gamma_{1}(V) \leq \rho + 1$. By Proposition \ref{shift-stuff} either $\shift{V}{}$ is $\cofi{0}$-acyclic or $\reg(\shift{V}{}) \leq \rho-1$. In any case, evaluating the long exact sequence \cite[Theorem 1]{gan-shift-seq} at a set of size $a+\rho$ for a fixed $a \geq 1$, we get an exact sequence
\begin{align*}
 \cofi{a+1}(\shift{V}{})_{a+\rho} \rarr \shift{\!\cofi{a+1}(V)}{}_{a+\rho} \rarr \cofi{a}(V)_{a+\rho} \rarr \cofi{a}(\shift{V}{})_{a+\rho} = 0
\end{align*}
of $\comm\sym{a+\rho}$-modules. Therefore by Proposition \ref{nu-props} we have 
\begin{align*}
 \nu_{I}\!\left(
  \cofi{a}(V)_{a+\rho}
 \right) &\geq 
 \nu_{I}\!\left(
 \shift{\!\cofi{a+1}(V)}{}_{a+\rho}
 \right) \, ,
\end{align*}
and hence
\begin{align*}
 \gamma_{a}(V) &\geq 
 \nu_{I}\!\left(
 \Res_{\sym{a+\rho}}^{\sym{a+\rho+1}}\!\left(
 \cofi{a+1}(V)_{a+\rho+1}
 \right)
 \right) - a
 \geq 
 \nu_{I}\!\left(
 \cofi{a+1}(V)_{a+\rho+1}
 \right) - a - 1 = \gamma_{a+1}(V)
\end{align*}
for every $a \geq 1$, establishing the weakly decreasing sequence. We know the sequence stabilizes at $\gamma$ by Theorem \ref{nss-nu}. We will consider two cases: 
\begin{itemize}
 \item $\shift{V}{}$ is $\cofi{0}$-acyclic or $\reg(\shift{V}{}) < \rho-1$: here $\cofi{a+1}(\shift{V}{})_{a+\rho} = 0$ so we get an isomorphism
\begin{align*}
 \shift{\!\cofi{a+1}(V)}{}_{a+\rho} \cong \cofi{a}(V)_{a+\rho} \, ,
\end{align*}
of $\comm\sym{a+\rho}$-modules. Thus, invoking the full strength of  Proposition \ref{nu-props} part (3), we get
\begin{align*}
 \gamma_{a}(V) &= 
 \nu_{I}\!\left(
 \Res_{\sym{a+\rho}}^{\sym{a+\rho+1}}\!\left(
 \cofi{a+1}(V)_{a+\rho+1}
 \right)
 \right) - a
 \\
 &=
 \begin{cases}
 \gamma_{a+1}(V) + 1 & \text{if $\gamma_{a+1}(V) = \rho - \frac{a+\rho+1}{\ell}$,}
 \\
 \gamma_{a+1}(V) &\text{otherwise.}
 \end{cases}
\end{align*}
If $\gamma_{1}(V) \neq \gamma$, then the maximum index $a \geq 1$ with $\gamma_{a}(V) > \gamma$ will satisfy $\gamma_{a+1}(V) = \gamma$, and by the above will also satisfy
\begin{align*}
 \gamma_{a}(V) = \rho+1 - \frac{a+\rho+1}{\ell} \quad
 \text{and} \quad
 \gamma = \gamma_{a+1}(V) = \rho - \frac{a+\rho+1}{\ell}
\end{align*}
But here $\rho - \gamma = h^{\gamma}(V) = 0$ by \cite[Corollary 2.13]{cmnr-range} and Proposition \ref{shift-stuff}, making the above impossible. Thus $\gamma_{1}(V) = \gamma$ is forced.
\item $\shift{V}{}$ is not $\cofi{0}$-acyclic and $\reg(\shift{V}{}) = \rho-1$: here by Proposition \ref{shift-stuff} we have $\crit(\shift{V}{}) = \gamma$. Applying Proposition \ref{nu-props} part (1) to the exact sequence
\begin{align*}
 \cofi{a+1}(\shift{V}{})_{a+\rho} \rarr \shift{\!\cofi{a+1}(V)}{}_{a+\rho} \rarr \cofi{a}(V)_{a+\rho}
\end{align*}
of $\comm\sym{a+\rho}$-modules, we have 
\begin{align*}
  \min\!\left\{
  \nu_{I}\!\left(\cofi{a+1}(\shift{V}{})_{a+\rho}\right),\,
  \nu_{I}\!\left(\cofi{a}(V)_{a+\rho}\right)
 \right\}
 \leq 
 \nu_{I}\!\left(
  \shift{\!\cofi{a+1}(V)}{}_{a+\rho}
 \right)
 \\
  \min\!\left\{
  \gamma_{a+1}(\shift{V}{}) + 1,\,
  \gamma_{a}(V)
 \right\}
 \leq 
 \begin{cases}
 \gamma_{a+1}(V) + 1 & \text{if $\gamma_{a+1}(V) = \rho -
 \frac{a+\rho+1}{\ell}$,}
 \\
 \gamma_{a+1}(V) &\text{otherwise.}
 \end{cases}
\end{align*}
If $\gamma_{1}(V) \neq \gamma$, then the maximum index $a \geq 1$ with $\gamma_{a}(V) > \gamma$ satisfies $\gamma_{a+1}(V) = \gamma$, and by the above also satisfies
\begin{align*}
  \min\!\left\{
  \gamma_{a+1}(\shift{V}{}) + 1,\,
  \gamma_{a}(V)
 \right\}
 \leq 
 \begin{cases}
 \gamma+1 & \text{if $\gamma = \rho -
 \frac{a+\rho+1}{\ell}$,}
 \\
 \gamma &\text{otherwise.}
 \end{cases}
\end{align*}
Here we know $\gamma_{a+1}(\shift{V}{}) + 1 \geq \gamma + 1$, making the left hand side strictly bigger than $\gamma$, forcing 
\begin{align*}
 \gamma &= \rho - \frac{a+\rho+1}{\ell} 
 \\
 a + \rho + 1 &= \ell(\rho-\gamma) = \ell h^{\gamma}(V)
 \\
 a &= \ell h^{\gamma}(V) - \rho - 1 = (\ell-1) h^{\gamma}(V) - \gamma - 1 \, .
\end{align*}
\end{itemize}
\end{proof}

%

\begin{cor} \label{good-ideal-range}
 Let $\comm$ be a commutative ring, $\ell \geq 2$, and assume $\comm\sym{\ell}$ has a good ideal. Then for $V \colon \FI \rarr \lMod{\comm}$ presented in finite degrees which is not $\cofi{0}$-acyclic, writing $\gamma := \crit(V)$, $\rho := \reg(V)$, we have 
$t_{a}(V) - a = \rho$ whenever
\begin{align*}
  a \geq \max\{1,(\ell-1)h^{\gamma}(V)-\gamma\} \, .
\end{align*}
\end{cor}
\begin{proof}
 Writing $I \subseteq \comm\sym{\ell}$ for a choice of good ideal, for the non-vanishing $\cofi{a}(V)_{a + \rho} \neq 0$, it suffices to have
\begin{align*}
 \nu_{I}\!\left(\cofi{a}(V)_{a + \rho}\right) \leq a+\rho \, ,
\end{align*}
which is guaranteed in the asserted range by Theorem \ref{sharp-nu} since $\gamma \leq \rho$.
\end{proof}

\subsection{Localization} 
As in \cite[Section 4]{nss-regularity}, we shall prove Theorem \ref{main} by reducing the general case to Corollary \ref{good-ideal-range} via localization so that good ideals are available, namely \cite[Propositions 3.1, 3.2]{nss-regularity}. We include preliminaries about how the homological invariants of $\FI$-modules interact with localization.

Given a multiplicatively closed subset $\Omega$ of a commutative ring $\comm$, we write $\Omega^{-1}\comm$ for the corresponding localization. Given $V \colon \FI \rarr \lMod{\comm}$, we write 
\begin{align*}
 \Omega^{-1} V \colon \FI \rarr \lMod{\Omega^{-1}\comm}
\end{align*}
for $V$ composed with the localization functor $\Omega^{-1} \colon \lMod{\comm} \rarr \lMod{\Omega^{-1}\comm}$. Note that $\Omega^{-1}$ is the left adjoint of the inclusion $\lMod{\Omega^{-1}\comm} \emb \lMod{\comm}$. The assignment $V \mapsto \Omega^{-1}V$ is itself (with some abuse of notation) a functor 
\begin{align*}
 \Omega^{-1} \colon [\FI,\,\lMod{\comm}] \rarr [\FI,\,\lMod{\Omega^{-1}\comm}]
\end{align*}
which is a left adjoint. If $\Omega = \{\ell^{n} :n\in \nn\}$ for some $\ell \in \comm$, we write $\locl{\comm} := \Omega^{-1}\comm$ and $\locl{V} := \Omega^{-1}V$.

\begin{lem} \label{tort}
Let $\comm$ be a commutative ring and $M$ be a $\comm$-module. Then $M = 0$ if and only if both $\loct{M} = 0$ and $\locr{M} = 0$.
\end{lem}
\begin{proof}
For the nontrivial implication, suppose $\loct{M} = 0$ and $\locr{M} = 0$. Fix $x \in M$. Because $x$ lies in the kernel of $M\rarr \loct{M}$, there exists $a \in \nn$ such that $2^{a}x = 0$, and similarly there exists $b \in \nn$ such that $3^{b}x = 0$. Thus $x = 0$ because $2^{a}$ and $3^{b}$ are coprime.
\end{proof}

\begin{cor} \label{loc-deg}
 Given a commutative ring $\comm$ and $W \colon \FI \rarr \lMod{\comm}$, 
\begin{align*}
 \deg(W) = \max\!\left\{
 \deg\!\left(\loct{W}\right),\,
 \deg\!\left(\locr{W}\right) \,
 \right\} 
\end{align*}
\end{cor}
\begin{proof}
By the equalities $\left(\loct{W}\right)_{\!S} = \loct{W_{S}}$, $\left(\locr{W}\right)_{\!S} = \locr{W_{S}}$ and Lemma \ref{tort}, 
\begin{align*}
 \{d \geq -1 : W_{S} = 0 \text{ for } |S| > d\}
\end{align*}
 is equal to
\begin{align*}
 \left\{
 d \geq -1 :\! \left(\loct{W}\right)_{\!S} \!= 0 \text{ for } |S| > d
 \right\} \cap
 \left\{
 d \geq -1 :\! \left(\locr{W}\right)_{\!S} \!= 0 \text{ for } |S| > d
 \right\} 
\end{align*}
as subsets of $\{-1,0,1,2,3,\dots\} \cup \{\infty\}$.
Now observe that all three sets in this relation are rays bounded below, that is, writing $a := \deg(W)$, $ b := \deg\!\left(\loct{W}\right)$, $c := \deg\!\left(\locr{W}\right)$, we have
\begin{align*}
 \{d \geq -1 : W_{S} = 0 \text{ for } |S| > d\} &= [a,\infty] \, ,
 \\
 \left\{
 d \geq -1 :\! \left(\loct{W}\right)_{\!S} \!= 0 \text{ for } |S| > d
 \right\} &= [b,\infty] \, ,
 \\
\left\{
 d \geq -1 :\! \left(\locr{W}\right)_{\!S} \!= 0 \text{ for } |S| > d
 \right\} &= [c,\infty] \, . 
\end{align*}
Now from the equality $[a,\infty] = [b,\infty] \cap [c,\infty]$ we deduce $a = \max\{b,c\}$ as desired.
\end{proof}

\begin{prop} \label{cofi-commute}
 Given a commutative ring $\comm$ with a multiplicatively closed subset $\Omega \subseteq \comm$ and $V \colon \FI \rarr \lMod{\comm}$, we have 
\begin{align*}
 \cofi{i}(\Omega^{-1}V) \cong \Omega^{-1}\!\cofi{i}(V)
\end{align*}
for every $i \geq 0$.
\end{prop}
\begin{proof}
 Let $P_{\star} \rarr V$ be a projective resolution of $V$ in the abelian category $[\FI,\lMod{\comm}]$. Because the localization functor 
\begin{align*}
 \Omega^{-1} \colon [\FI,\,\lMod{\comm}] \rarr [\FI,\,\lMod{\Omega^{-1}\comm}]
\end{align*}
is left adjoint to an exact functor (hence preserves projectives) and is exact itself, $\Omega^{-1}P_{\star} \rarr \Omega^{-1}V$ is a projective resolution of $\Omega^{-1}V$ in $[\FI,\,\lMod{\Omega^{-1}\comm}]$. Therefore  
\begin{align*}
 \cofi{0}(\Omega^{-1}V)_{S} 
 &= \coker\! \left(
 \bigoplus_{A \subsetneq S} \Omega^{-1}V_{A} \rarr \Omega^{-1}V_{S}
 \right)
 \\ 
 &\cong \Omega^{-1}\! \coker \left(
 \bigoplus_{A \subsetneq S} V_{A} \rarr V_{S}
 \right) = \Omega^{-1}\cofi{0}(V)_{S}
\end{align*}
for every finite set $S$ and consequently $\cofi{0}(\Omega^{-1}V) \cong \Omega^{-1}\cofi{0}(V)$. Finally, 
\begin{align*}
 \cofi{i}(\Omega^{-1}V) &\cong \co_{i}(\cofi{0}(\Omega^{-1}P_{\star}))
 \\
 &\cong \co_{i}(\Omega^{-1}\cofi{0}(P_{\star}))
 \\
 &\cong \Omega^{-1}\!\co_{i}(\cofi{0}(P_{\star})) 
 \cong \Omega^{-1}\!\cofi{i}(V)
\end{align*}
for every $i \geq 0$. 
\end{proof}

\begin{cor} \label{loc-cofi}
 Given a commutative ring $\comm$, $i \geq 0$, and $V \colon \FI \rarr \lMod{\comm}$, the following hold:
\begin{birki}
 \item For every $i \geq 0$, we have 
 $t_{i}(V) = \max\left\{
 t_{i}\!\left(\loct{V}\right),\,
 t_{i}\!\left(\locr{V}\right)
 \right\}$.
 \vspace{0.1cm}
 \item  
 $\reg(V) = \max\!\left\{
 \reg\!\left(\loct{V}\right),\,
 \reg\!\left(\locr{V}\right)
 \right\}$. 
\end{birki}

\end{cor}
\begin{proof}
%
By taking $W = \cofi{i}(V)$ in part (2) of Corollary \ref{loc-deg} and Proposition \ref{cofi-commute}, 
\begin{align*}
 t_{i}(V) &= \deg (\cofi{i}(V)) 
 \\
 &= \max\!\left\{
 \deg\!\left(\loct{\cofi{i}(V)}\right),\,
 \deg\!\left(\locr{\cofi{i}(V)}\right) \,
 \right\} 
 \\
 &= \max\!\left\{
 \deg\!\left(\cofi{i}\!\left(\loct{V}\right)\right),\,
 \deg\!\left(\cofi{i}\!\left(\locr{V}\right)\right) \,
 \right\} 
 \\
 &= \max\!\left\{t_{i}\!\left(\loct{V}\right),\,
 t_{i}\!\left(\locr{V}\right)
 \right\} \, ,
\end{align*}
and (1) follows. Then (2) follows immediately from (1).
\end{proof}

\begin{prop} \label{locoh-commute}
 Given a commutative ring $\comm$ with a multiplicatively closed subset $\Omega \subseteq \comm$ and $V \colon \FI \rarr \lMod{\comm}$ presented in finite degrees, then $\Omega^{-1}V \colon \FI \rarr \lMod{\Omega^{-1}\comm}$ is also presented in finite degrees such that
\begin{align*}
 \locoh{j}(\Omega^{-1}V) \cong \Omega^{-1}\!\locoh{j}(V)
\end{align*}
for every $j \geq 0$.
\end{prop}
\begin{proof}
 By \cite[Theorem 2.10]{cmnr-range}, there exists a cochain complex 
\begin{align*}
 I^{\star} : 0 \rarr I^{0} \rarr I^{1} \rarr \cdots \rarr I^{M} \rarr 0
\end{align*}
in $[\FI,\lMod{\comm}]$ such that $I^{\star}$ is exact in all high enough degrees, $I^{0} = V$, for $j \geq 1$ $I^{j}$ is $\cofi{0}$-acyclic and presented in finite degrees. Now the cochain complex $\Omega^{-1}I^{\star}$ in $[\FI,\lMod{\Omega^{-1}\comm}]$ satisfies 
\begin{itemize}
 \item $(\Omega^{-1}I)^{0} = \Omega^{-1}I^{0} = \Omega^{-1}V$,
 \item $\Omega^{-1}I^{\star}$ is exact in all high enough degrees,
 \item for every $j \geq 0$, $(\Omega^{-1}I)^{j} = \Omega^{-1}I^{j}$ is presented in finite degrees by Theorem \ref{cofi-commute}, 
 \item for every $j \geq 1$, $\Omega^{-1}I^{j}$ is $\cofi{0}$-acyclic by Theorem \ref{cofi-commute} (hence is \emph{semi-induced} in the sense of $\cite{cmnr-range}$, see \cite[Theorem 2.4]{bahran-reg}),
\end{itemize}
Therefore again by \cite[Theorem 2.10]{cmnr-range} and the exactness of $\Omega^{-1}$, for every $j \geq 0$ we have 
\begin{align*}
 \locoh{j}(\Omega^{-1}V) \cong \co^{j}(\Omega^{-1}I^{\star}) \cong \Omega^{-1}\!\co^{j}(I^{\star}) \cong \Omega^{-1}\!\locoh{j}(V) \, .
\end{align*}
\end{proof}

\begin{cor} \label{loc-locoh}
 Given a commutative ring $\comm$ and $V \colon \FI \rarr \lMod{\comm}$ presented in finite degrees, then the following hold:
\begin{birki}
 \item $\locl{V} \colon \FI \rarr \lMod{\locl{\comm}}$ is presented in finite degrees for $\ell \in \{2,3\}$.
\vspace{0.1cm}
\item For every $j \geq 0$, we have $
 h^{j}(V) = \max\!\left\{
 h^{j}\!\left(\loct{V}\right),\,
 h^{j}\!\left(\locr{V}\right)
 \right\}$.
\vspace{0.1cm}
\item If $V$ is not $\cofi{0}$-acyclic, setting $\gamma := \crit(V)$, there exists $\ell \in \{2,3\}$ such that $\locl{V}$ is not $\cofi{0}$-acyclic, $h^{\gamma}(\locl{V}) = h^{\gamma}(V)$, $\reg(\locl{V}) = \reg(V)$, and $\crit(\locl{V}) = \gamma$.
\end{birki}
 
\end{cor}
\begin{proof}
(1) follows from Corollary \ref{loc-cofi}. Taking $W = \cofi{i}(V)$ in Corollary \ref{loc-deg} and Proposition \ref{locoh-commute}, 
\begin{align*}
 h^{j}(V) &= \deg (\locoh{j}(V)) 
 \\
 &= \max\!\left\{
 \deg\!\left(\loct{\locoh{j}(V)}\right),\,
 \deg\!\left(\locr{\locoh{j}(V)}\right) \,
 \right\} 
 \\
 &= \max\!\left\{
 \deg\!\left(\locoh{j}\!\left(\loct{V}\right)\right),\,
 \deg\!\left(\locoh{j}\!\left(\locr{V}\right)\right) \,
 \right\} 
 \\
 &= \max\!\left\{h^{j}\!\left(\loct{V}\right),\,
 h^{j}\!\left(\locr{V}\right)
 \right\} \, ,
\end{align*}
establishing (2). For (3), let us also write 
\begin{align*}
 \rho := \reg(V)\,, \quad
 \gamma[2] := \crit\!\left(\loct{V}\right)\,, \quad
 \gamma[3] := \crit\!\left(\locr{V}\right)\,.
\end{align*}
By part (2), and there is $\ell \in \{2,3\}$ such that $h^{\gamma}(\locl{V}) = h^{\gamma}(V) \geq 0$ and hence by \cite[Theorem 2.4]{bahran-reg}, $\locl{V}$ is not $\cofi{0}$-acyclic. By part (1) we know $\locl{V}$ is presented in finite degrees, so Theorem \ref{nss-main} and part (2) of Corollary \ref{loc-cofi} yield
\begin{align*}
 0 \leq \rho &= h^{\gamma}(V) + \gamma 
 = h^{\gamma}\!\left(\locl{V}\right) + \gamma \leq \reg\!\left(
 \locl{V}
 \right) \leq \rho \, .
\end{align*}
Therefore $h^{\gamma}\!\left(\locl{V}\right) + \gamma = \reg\!\left(
 \locl{V}\right) = \rho$ and hence $\gamma[\ell] \leq \gamma$ by the definition of $\crit(V[\ell]) = \gamma[\ell]$. If we had $\gamma[\ell] < \gamma$, the definition of $\crit(V) = \gamma$ and part (2) would give 
\begin{align*}
 \reg\!\left(
 \locl{V}\right) =
 \rho > h^{\gamma[\ell]}(V) + \gamma[\ell] \geq h^{\gamma[\ell]}\!\left(\locl{V}\right) + \gamma[\ell] = \reg\!\left(
 \locl{V}\right) \, ,
\end{align*}
a contradiction. Thus $\gamma[\ell] = \gamma$.
\end{proof}

\begin{proof}[Proof of \textbf{\emph{Theorem \ref{main}}}] Let us write $\rho := \reg(V)$ and $\gamma := \crit(V)$ as usual. If $2 \in \comm^{\times}$ is invertible, then $\comm\sym{2}$ has a good ideal by \cite[Proposition 3.1]{nss-regularity} and we can apply Corollary \ref{good-ideal-range}. 

For a general $\comm$, pick $\ell \in \{2,3\}$ as in part (3) of Corollary \ref{loc-locoh}, so that 
\begin{itemize}
 \item $h^{\gamma}(\locl{V}) = h^{\gamma}(V) \geq 0$,
 \vspace{0.1cm}
 \item $\reg(\locl{V}) = \rho$,
 \vspace{0.1cm}
 \item $\crit(\locl{V}) = \gamma$.
\end{itemize}
Now $\locl{\comm}\!\sym{\ell}$ has a good ideal by \cite[Proposition 3.1]{nss-regularity} if $\ell =2$, and by \cite[Proposition 3.2]{nss-regularity} if $\ell =3$. Now by part (1) of Corollary \ref{loc-cofi} and Corollary \ref{good-ideal-range}, we have 
\begin{align*}
 t_{a}(V) - a \geq t_{a}\!\left(\locl{V}\right) - a = \reg\!\left(\locl{V}\right) = \rho
\end{align*}
whenever 
\begin{align*}
 a &\geq \max\!\left\{1,(\ell-1)h^{\gamma}\!\left(\locl{V}\right) - \gamma\right\} = 
 \max\!\left\{1,(\ell-1)h^{\gamma}(V) - \gamma\right\}
\end{align*}
and hence whenever $a \geq \max\!\left\{1,2h^{\gamma}\!\left(V\right) - \gamma\right\}$ as $\ell \in \{2,3\}$.
Thus by the weakly increasing property in Theorem \ref{nss-main}, $t_{a}(V) - a = \rho$ in the same range. 
\end{proof}

\bibliographystyle{hamsalpha}
\bibliography{stable-boy}

\end{document}